\documentclass[a4paper,11pt,reqno]{amsart}

\usepackage[utf8]{inputenc}

\usepackage[T1]{fontenc}
\usepackage{lmodern}

\usepackage[colorinlistoftodos,bordercolor=orange,backgroundcolor=orange!20,linecolor=orange,textsize=normalsize]{todonotes}

\usepackage[foot]{amsaddr}
\usepackage{amsmath}  
\usepackage{amssymb}     
\usepackage{bbm}

\usepackage{bm}
\usepackage{hyperref}
\usepackage{mathrsfs}
\usepackage{mathtools}
\usepackage[inline]{enumitem} 
\usepackage{fixmath}
\usepackage{diffcoeff}
\usepackage{array}

\usepackage[capitalize]{cleveref}

\Crefname{figure}{Figure}{Figures}

\crefformat{equation}{\textup{#2(#1)#3}}
\crefrangeformat{equation}{\textup{#3(#1)#4--#5(#2)#6}}
\crefmultiformat{equation}{\textup{#2(#1)#3}}{ and \textup{#2(#1)#3}}
{, \textup{#2(#1)#3}}{, and \textup{#2(#1)#3}}
\crefrangemultiformat{equation}{\textup{#3(#1)#4--#5(#2)#6}}%
{ and \textup{#3(#1)#4--#5(#2)#6}}{, \textup{#3(#1)#4--#5(#2)#6}}{, and \textup{#3(#1)#4--#5(#2)#6}}

\Crefformat{equation}{#2Equation~\textup{(#1)}#3}
\Crefrangeformat{equation}{Equations~\textup{#3(#1)#4--#5(#2)#6}}
\Crefmultiformat{equation}{Equations~\textup{#2(#1)#3}}{ and \textup{#2(#1)#3}}
{, \textup{#2(#1)#3}}{, and \textup{#2(#1)#3}}
\Crefrangemultiformat{equation}{Equations~\textup{#3(#1)#4--#5(#2)#6}}%
{ and \textup{#3(#1)#4--#5(#2)#6}}{, \textup{#3(#1)#4--#5(#2)#6}}{, and \textup{#3(#1)#4--#5(#2)#6}}

\usepackage{tikz}

\usetikzlibrary{arrows}
\usetikzlibrary{patterns}
\usetikzlibrary{calc}
\usetikzlibrary{shapes}
\usetikzlibrary{positioning}

\usepackage[scr=boondox,scrscaled=1.05]{mathalfa}

\renewcommand{\geq}{\geqslant}
\renewcommand{\leq}{\leqslant}

\renewcommand{\le}{\leq}
\renewcommand{\ge}{\geq}

\newcommand{\arxiv}[1]{\href{http://arxiv.org/abs/#1}{arXiv:#1}}

\DeclareMathOperator{\sign}{\mathsf{sign}}

\DeclareMathAlphabet{\mathbfcal}{OMS}{cmsy}{b}{n}
\DeclareMathAlphabet{\mathbbold}{U}{bbold}{m}{n}

\newcommand{\C}{\mathbb{C}}
\newcommand{\Z}{\mathbb{Z}}
\newcommand{\N}{\mathbb{N}} 
\newcommand{\R}{\mathbb{R}}    
\newcommand{\Q}{\mathbb{Q}}

\newcommand{\puiseux}{\C\{\!\{x\}\!\}}

\newcommand{\posN}{\mathbb{N}^{*}}

\newcommand{\trop}[1][]{\ifthenelse{\equal{#1}{}}{ \mathbb{T} }{ \mathbb{T}(#1) }}

\newcommand{\abs}[1]{|{#1}|}

\DeclareMathOperator*{\val}{\mathsf{val}}

\newcommand{\card}[1]{|{#1}|}

\newtheorem{question}{Question}
\newtheorem{theorem}{Theorem}[section]
\newtheorem{proposition}[theorem]{Proposition}

\newtheorem{conjecture}[theorem]{Conjecture}
\newtheorem{lemma}[theorem]{Lemma}

\theoremstyle{definition}
\newtheorem{definition}[theorem]{Definition}

\theoremstyle{remark}
\newtheorem{remark}[theorem]{Remark}
\newtheorem{example}[theorem]{Example}

\tikzset{grid/.style={gray!30,very thin}}
\tikzset{axis/.style={gray!50,->,>=stealth'}}
\tikzset{convex/.style={draw=none,fill=lightgray,fill opacity=0.7}}
\tikzset{convexborder/.style={very thick}}
\tikzset{point/.style={blue!50}}
\tikzset{hs/.style={fill opacity=0.3,fill=orange,draw=none}}
\tikzset{hsborder/.style={orange,ultra thick,dashdotted}}

\newcommand{\overbar}[1]{\mkern 1.5mu\overline{\mkern-1.5mu#1\mkern-1.5mu}\mkern 1.5mu}

\newcommand{\trF}{\overbar{F}}
\newcommand{\trG}{\overbar{G}}
\newcommand{\trH}{\overbar{H}}

\newcommand{\pser}{S}
\newcommand{\pserii}{T}
\newcommand{\vring}{\mathcal{O}}
\newcommand{\lring}{O}

\newcommand{\wron}{W}
\newcommand{\mint}{I}

\newcommand{\bratf}{\C(x,y)}
\newcommand{\powseq}{\mathcal{S}}
\newcommand{\seqit}{s}

\newcommand{\powcon}{\xi}

\newcommand{\psupport}{\Lambda}

\newcommand{\sym}{\mathrm{Sym}}

\newcommand{\rootpol}{R}
\newcommand{\auxrootpol}{\overbar{R}}

\newcommand{\wentpol}{P}
\newcommand{\wrpol}{Q}

\newcommand{\degy}{\delta}

\makeatletter
\def\env@dmatrix{\hskip -\arraycolsep
  \let\@ifnextchar\new@ifnextchar
  \array{*\c@MaxMatrixCols{>{\displaystyle}c}}}
  
\newenvironment{bdmatrix}
  {\left[\env@dmatrix}
  {\endmatrix\right]}
\makeatother

\newcommand\vp{\ensuremath{\mathsf{VP}}}
\newcommand\vnp{\ensuremath{\mathsf{VNP}}}

\begin{document}

\title[Intersection multiplicity of a sparse and a low-degree curve]{Intersection multiplicity of a sparse curve and a low-degree curve}

\author[Pascal Koiran]{Pascal Koiran}
\author[Mateusz Skomra]{Mateusz Skomra}
\address{Univ Lyon, EnsL, UCBL, CNRS,  LIP, F-69342, LYON Cedex 07, France.}
  \email{firstname.lastname@ens-lyon.fr}

\begin{abstract}
  Let $F(x, y) \in \C[x,y]$ be a polynomial of degree $d$
  and let $G(x,y) \in \C[x,y]$  be a polynomial with $t$ monomials.
  We want to estimate the maximal multiplicity of a solution of
  the system $F(x,y) = G(x,y) = 0$. 
  Our main result is that the multiplicity of any isolated solution
  $(a,b) \in \C^2$ with nonzero coordinates is no greater than
  $\frac{5}{2}d^2t^2$. We ask whether this intersection multiplicity can be polynomially
  bounded in the number of monomials of $F$ and $G$, and we briefly review
  some connections between sparse polynomials and algebraic complexity theory.
\end{abstract}

\maketitle

\section{Introduction}

In this paper we consider the following problem. Let $F(x, y) \in \C[x,y]$ and $G(x,y) \in \C[x,y]$ be two polynomials with complex coefficients such that $F$ has degree $d \ge 1$ and $G$ has $t \ge 1$ monomials. We want to estimate the maximal multiplicity of an isolated solution of the system
\begin{equation}\label{eq:system}
F(x,y) = G(x,y) = 0 \, .
\end{equation}
Our main result is the following theorem.
\begin{theorem}\label{th:main}
  Suppose that $p \coloneqq (a, b) \in (\C \setminus \{0\})^{2}$ is an isolated solution of system~\cref{eq:system}. Then, the intersection multiplicity of $F(x,y)$ and $G(x,y)$ at $p$ is at most $\frac{5}{2}d^2t^2$.
\end{theorem}
The assumption that $a$ and $b$ are nonzero is crucial, as shown by the following examples.

\begin{example}
Let $F(x,y) \coloneqq x - y$ and $G(x, y) \coloneqq x^{2n} - y^{n}$. Then, $(0,0)$ is a solution of~\cref{eq:system} and its multiplicity is equal to $n$. Similarly, let $F(x,y) \coloneqq x - 1$ and $G(x, y) \coloneqq y^{n} + x - 1$. Then, $(1,0)$ is a solution of~\cref{eq:system} and its multiplicity is equal to $n$. In \cref{th:main} the restriction to points $p$ with nonzero coordinates 
is therefore unavoidable.
\end{example}

A polynomial bound on the number of real zeros of a system of the same form
was obtained in~\cite{koiran_sparse_curve}: the number of
real isolated solutions of~\cref{eq:system}  is $O(d^3t+d^2t^3)$.
More generally, this bound applies to the number of connected components
of the set of real solutions.
\Cref{th:main} can therefore be viewed as an analogue for intersection
multiplicity of this result from~\cite{koiran_sparse_curve}.
Both results belong to {\em fewnomial theory}, which seeks quantitative bounds
on polynomial systems\footnote{More general functions than polynomials can sometimes be allowed, e.g., the exponential and logarithmic functions,  or more generally Pfaffian functions.}
in terms of the number of nonzero monomials occurring
in the system. Historically, quantitative bounds were first obtained
in terms of the degrees of the polynomials involved instead of the
number of monomials. For instance,  B\'ezout's theorem shows that
$\deg(F) \cdot \deg(G)$
is an upper bound on the intersection multiplicity
of any isolated solution of the system (the same bound of course
applies in fact to
the sum of intersection multiplicities of all isolated solutions).
The bound in \cref{th:main} is of a mixed form since it involves the number of monomials of $G$ but the degree of $F$.
It is natural to ask for a bound that depends only on the number of monomials
in $F$ and $G$. We therefore highlight the following question.
\begin{question} \label{q1}
  Let $F, G \in \C[x,y]$ be two polynomials with at most $t$ monomials each.
  What is the maximal multiplicity of an isolated solution
  $p = (a, b) \in (\C \setminus \{0\})^{2}$ of system~\cref{eq:system}?
  In particular, is the multiplicity of $p$ polynomially bounded
  in $t$, i.e., bounded from above by $t^c$
  where $c$ is some absolute constant?
\end{question}

The first focus of fewnomial theory~\cite{Khov91,sottile11}
was on the number of real solutions
of multivariate systems.
In particular, a seminal result by Khovanskii~\cite{Khov91}
shows that a system of $n$
polynomials in $n$ variables involving $l+n+1$ distinct monomials has
less than 
\begin{align} \label{khovanskii}
  2^{\binom {l+n}2}(n+1)^{l+n}
\end{align} 
non-degenerate positive solutions. 
This bound was improved by  Bihan and Sottile~\cite{BiSot} to
 \[
 \frac{e^2+3}{4}2^{\binom l2}n^l.
 \]
These results can be viewed as far reaching generalizations of Descartes' rule
of signs, which implies that a univariate polynomial with $t$ monomials
has at most $t-1$ positive roots.
As pointed out in~\cite{koiran_sparse_curve},
the analogue of \cref{q1} for real
roots is very much open: it is not known whether the number of isolated
real solutions of a system $F(x,y)=G(x,y)=0$ is polynomially bounded in the
number of monomials of $F$ and $G$.

The first result on fewnomials and multiplicities seems to be an analogue of
Descartes' rule due to Haj{\'o}s (see~\cite{hajos53,Lenstra99b}
and \cref{le:hajos} below): the multiplicity
of any nonzero root of a univariate polynomial $f \in \C[X]$ with $t$ monomials 
is at most $t-1$.
For multivariate systems, an analogue of Khovanskii's bound~\cref{khovanskii}
was obtained by Gabrielov~\cite{gabrielov_multiplicity}.
He showed that for a system of $n$ polynomials in $n$ variables
involving at most $t$ monomials, the multiplicity of any solution in
$(\C \setminus \{0\})^n$
does not exceed
\[
2^{t(t-1)/2}[\min(n,t)+1]^t.
\]
In particular, this provides a $2^{O(t^2)}$ upper bound for \cref{q1}. Gabrielov's result also implies an exponential bound for the problem
  considered in \cref{th:main} instead of our polynomial bound.
After~\cite{gabrielov_multiplicity}, subsequent work has  focused on
multiplicity estimates for more general (``Noetherian'') multivariate systems,
see, e.g.,~\cite{gab98,binyamini15}.

It is easily seen that the Haj{\'o}s lemma is tight,
  but nevertheless proving tight bounds for ``structured'' univariate
  polynomials may be challenging.
  As an example, we propose the following question
  (we consider in \cref{complexity} more general structured
  families of polynomials).
  \begin{question} \label{fgplus1}
    Let $f,g \in \C[X]$ be two univariate polynomials with at most~$t$ monomials
    each. What is the maximal multiplicity of a nonzero root
    of the polynomial $f(x)g(x)+1$?
    In particular, is there an $o(t^2)$ bound on this maximal multiplicity?
  \end{question}
  Note that the Haj{\'o}s lemma yields $t^2$ as an upper bound for the maximal
  multiplicity. We note also that this question can be cast as a question
  on bivariate systems of the form~\cref{eq:system}
  with at most $t+1$ monomials
  each, namely:
  $F(x,y)\coloneqq y-f(x)$, $G(x,y) \coloneqq g(x)y+1$.
  Indeed, the multiplicity of any root~$a$ of $fg+1$
  is equal to the multiplicity of $(a,f(a))$ as a root of this bivariate
  system (this follows from instance from \cref{le:halphen}
  below). It is not clear whether this more ``geometric'' formulation is
  useful to make progress on \cref{fgplus1}, though.

\subsection{Sparse polynomials in algebraic complexity} \label{complexity}

Obtaining effective bounds for sparse polynomials or sparse polynomial systems
is an interesting subject in its own right, but there is also a connection
to lower bounds in algebraic complexity.
 In particular, the following ``real $\tau$-conjecture''
was put forward in~\cite{koiran_real_tau} as a variation
on the original $\tau$-conjecture by Shub and Smale (Problem~4
in~\cite{SmaleProblems}).
\begin{conjecture}[real $\tau$-conjecture]  \label{realtau}
Consider a nonzero polynomial of the form  
\[
f(X)=\sum_{i=1}^k \prod_{j=1}^m f_{ij}(X),
\]
where each $f_{ij} \in \R[X]$ has at most $t$ monomials. The number of real roots of $f$ is bounded by a polynomial function of $kmt$.
\end{conjecture}
It was shown in~\cite{koiran_real_tau} that this conjecture implies the separation of
the complexity classes $\vp$ and \vnp.\footnote{The separation result derived
  in~\cite{koiran_real_tau} is actually a little weaker than $\vp \neq \vnp$;
  a proof that \cref{realtau} implies the full separation
  $\vp \neq \vnp$
  can be found in the PhD thesis by S\'ebastien Tavenas~\cite{tavenasphd}.
  In fact, a bound on the number of real roots that is polynomial
  in $kt2^m$ would suffice for that purpose
  \cite[Theorems~3.25 and~3.38]{tavenasphd}.}
See~\cite{grenet_koiran_powering,koiran_sparse_curve} for some partial results toward \cref{realtau} and applications to algebraic complexity. It was recently shown that
\cref{realtau} is true ``on average''~\cite{briquel_burgisser_average}.
For earlier work connecting ``sparse like'' polynomials to algebraic
complexity see~\cite{BC76,Gri82,Risler85}.
For an introduction to the $\vp$ versus $\vnp$ problem we recommend~\cite{burgisser_completeness}.
The authors' interest for intersection multiplicity was sparked by the following variation 
on \cref{realtau}:
\begin{conjecture}[$\tau$-conjecture for multiplicities] \label{multau}
Consider a nonzero polynomial of the form  
\[
f(X)=\sum_{i=1}^k \prod_{j=1}^m f_{ij}(X),
\]
where each $f_{ij} \in \C[X]$ has at most $t$ monomials.
The multiplicity of any nonzero complex root of $f$ is bounded by a polynomial function of $kmt$.
\end{conjecture}
The idea of looking at multiplicities in this context was introduced by Hrube\v{s}~\cite{hrubes_complex_tau}. \Cref{multau} implies a slightly weaker separation
than $\vp \neq \vnp$, which can be obtained under a bound
on multiplicities that is only polynomial in
$kt2^{m}$~\cite[Section~2.2]{tavenasphd}.
Finally we point out that there is also a ``$\tau$-conjecture
for Newton polygons,'' which implies the separation
$\vp \neq \vnp$~\cite{koiran_tau_newton}.
It was recently announced by Hrube\v{s}~\cite{hrubes_runners}
that \cref{realtau} implies
the $\tau$-conjecture for Newton polygons.
Moreover, it follows from~\cite{hrubes_complex_tau} that \cref{realtau}
also implies \cref{multau}. The real $\tau$-conjecture is therefore the
strongest of these 3 conjectures (and there is no known implication
between the other two).

\subsection{Outline of the proof}

In this section we present some of the ideas of the proof of \cref{th:main} in an informal way. The actual proof is presented in \cref{sec:proof}
after some preliminaries in \cref{sec:puiseux,sec:mult,sec:deriv_lemmas}.

Like in~\cite{koiran_sparse_curve} we rely heavily
on the properties of Wronskian determinants.
Let us assume first that the relation $F(x,y)=0$ can be inverted locally
in a neighborhood of $(a,b)$ as $y=\phi(x)$, where $\phi$ is an analytic
function. In this case we just have to bound the multiplicity of $a$
as a root of the univariate function
\[
G(x,\phi(x))=\sum_{\alpha \in \Lambda} c_{\alpha} x^{\alpha_1}\phi(x)^{\alpha_2} \, ,
\]
where the support $\Lambda$ of $G$ is of size $t$.
This multiplicity can be bounded with the help of the Wronskian
determinant of the $t$ functions $\{x^{\alpha_1}\phi(x)^{\alpha_2} \colon \alpha \in \Lambda\}$
(see \cref{pr:val_by_wron} and \cref{re:val_by_wron}).
The entries of the Wronskian determinant may be of very high degree due
to the presence of the exponents $\alpha_1,\alpha_2$, over which we have
no control. Fortunately, it turns out that high exponents can be factored
out and we can reduce to the case of a determinant with entries
of low degree in $x$ and $\phi(x)$.
We can then conclude by applying B{\'e}zout's theorem (\cref{th:bezout}) to
$F(x,y)=0$ and to a low-degree determinant.

The above proof idea is not always applicable since it might not be possible
to invert the relation $F(x,y)=0$ as $y=\phi(x)$. In particular, we must
explain how to handle the case where $(a,b)$ is a singular point of the
curve $F(x,y)=0$. It is well known that the behavior of an algebraic
curve near a singular point can be described with the help of Puiseux series
(they were invented for that purpose).
In the actual proof we therefore work with Puiseux series instead of
analytic functions, and we use a characterization of intersection
multiplicity in terms of Puiseux series (\cref{le:halphen}).

\section{Puiseux series, their derivatives, and Wronskians} \label{sec:puiseux}

\begin{definition}
A \emph{Puiseux series} is a formal series of the form
\begin{equation}\label{eq:puiseux}
\pser(x) = \sum_{i = 1}^{\infty}c_{i}x^{\lambda_{i}} \, ,
\end{equation}
where the coefficients 
$c_{i}$ are nonzero complex numbers and the exponents $(\lambda_{i})_{i \ge 1} \subset \Q^{\N}$ form a strictly increasing sequence of rational numbers with the same denominator. (We also allow the sum in~\cref{eq:puiseux} to be finite.) There is also a special empty series denoted by $0$.
\end{definition}
Puiseux series can be added and multiplied in the usual way. Moreover, it is well known that the set of Puiseux series forms an algebraically closed field (see, e.g.,~\cite[Chapter~IV, \S~3.2]{walker_algebraic_curves}). In this paper, we denote the field of Puiseux series by $\puiseux$.

\begin{definition}
Given a Puiseux series $\pser(x)$ as in~\cref{eq:puiseux} we define its \emph{valuation} $\val(\pser(x))$ as the lowest exponent of $\pser(x)$, i.e., $\val(\pser(x)) \coloneqq \lambda_{1}$. We use the convention that $\val(0) = +\infty$. We denote by $\vring \subset \puiseux$ the set of all Puiseux series with nonnegative valuation,
\[
\vring \coloneqq \{\pser(x) \in \puiseux \colon \val(\pser(x)) \ge 0 \} \, .
\]
\end{definition}

It is easy to check that the valuation map has the following two properties. For every pair of Puiseux series $\pser(x), \pserii(x) \in \puiseux$ we have
\begin{equation}\label{eq:val}
\begin{aligned}
\val(\pser(x)\pserii(x)) &= \val(\pser(x)) + \val(\pserii(x)) \, ,\\
\val(\pser(x) + \pserii(x)) &\ge \min\bigl(\val(\pser(x)), \val(\pserii(x)) \bigr) \, .
\end{aligned}
\end{equation}
In particular,~\cref{eq:val} shows that $\vring$ is a subring of $\puiseux$. This subring is called the \emph{valuation ring (of Puiseux series)}. We can now define the derivatives of Puiseux series and their Wronskians.

\begin{definition}
Given a Puiseux series $\pser(x)$ as in~\cref{eq:puiseux} we define its \emph{(formal) derivative} $\diffp{\pser}{x}(x) \in \puiseux$ as
\[
\diffp{\pser}{x}(x) \coloneqq \sum_{i = 1}^{\infty}\lambda_{i}c_{i}x^{\lambda_{i} - 1} \, .
\]
Similarly, for every $n \ge 1$, we denote by $\diffp[n]{\pser}{x} \in \puiseux$ the $n$th derivative of $\pser(x)$, i.e., the series obtained from $\pser(x)$ by deriving it $n$ times. We use the convention that $\diffp[0]{\pser}{x}(x) = \pser(x)$. To improve readability, we also use the notation $\diffp[n]{}{x}(\pser(x))$ instead of $\diffp[n]{\pser}{x}(x)$.
\end{definition}
It is easy to check that derivatives of Puiseux series satisfy the following natural properties. For every pair of Puiseux series $\pser(x), \pserii(x) \in \puiseux$ we have
\begin{equation}\label{eq:der}
\begin{aligned}
\diffp{(\pser + \pserii)}{x}(x) &= \diffp{\pser}{x}(x) +\diffp{\pserii}{x}(x) \, ,\\
\diffp{(\pser\pserii)}{x}(x) &= \diffp{\pser}{x}(x)\pserii(x) + \pser(x)\diffp{\pser}{x}(x) \, .
\end{aligned}
\end{equation}
Moreover, we note that for every Puiseux series $\pser(x) \in \puiseux$ we have the inequality
\begin{equation}\label{eq:val_dif}
\val\Bigl( \diffp{\pser}{x}(x) \Bigr) \ge \val(\pser(x)) - 1 \, .
\end{equation}
(The inequality is strict when $\val(\pser(x)) = 0$.)

\begin{definition}
If $\pser_{1}(x), \dots, \pser_{n}(x) \in \puiseux$ are Puiseux series, then we define their \emph{Wronskian}, denoted $\wron(\pser_{1}(x), \dots, \pser_{n}(x)) \in \puiseux$, as the determinant
\begin{equation}\label{eq:wron}
\wron(\pser_{1}(x), \dots, \pser_{n}(x)) \coloneqq 
\det \begin{bdmatrix}
\pser_{1}(x) & \pser_{2}(x) & \dots & \pser_{n}(x)\\[1.5ex]
\diffp{\pser_{1}}{x}(x) & \diffp{\pser_{2}}{x}(x) & \dots & \diffp{\pser_{n}}{x}(x)\\[1.5ex]
\vdots & \vdots & \vdots & \vdots\\[1.5ex]
\diffp[n-1]{\pser_{1}}{x}(x) & \diffp[n-1]{\pser_{2}}{x}(x) & \dots & \diffp[n-1]{\pser_{n}}{x}(x)
\end{bdmatrix} \, .
\end{equation}
\end{definition}

It is immediate to see that if $\pser_{1}(x), \dots, \pser_{n}(x)$ are linearly dependent over~$\C$, then their Wronskian is identically zero. 
B{\^o}cher~\cite{bocher_dependence} proved that the converse is true in the context of analytic functions.\footnote{An alternative proof for formal power series can be found in~\cite{BoDu}.} It is easy to check that the proof presented in~\cite{bocher_dependence} carries over to Puiseux series. (The proof is based on the fact that $\diffp{\pser}{x}(x) = 0$ implies $\pser(x) = c$ for some $c \in \C$ and this is true for both for analytic functions and Puiseux series.)

\begin{theorem}[{\cite{bocher_dependence}}]\label{th:bocher}
If $\pser_{1}(x), \dots, \pser_{n}(x) \in \puiseux$ are Puiseux series, then their Wronskian is equal to $0$ if and only if $\pser_{1}(x), \dots, \pser_{n}(x)$ are linearly dependent over $\C$. In other words, we have $\wron(\pser_{1}(x), \dots, \pser_{n}(x)) = 0$ if and only if there exist complex constants $c_{1}, \dots, c_{n} \in \C$, not all equal to $0$, such that $c_{1}\pser_{1}(x) + \dots + c_{n}\pser_{n}(x) = 0$.
\end{theorem}

In~\cite{voorhoeve_wronskian}, Wronskians are used to bound multiplicities of zeros of certain functions. The next proposition and its proof is an adaptation of \cite[Theorem~1]{voorhoeve_wronskian} to the context of Puiseux series.

\begin{proposition}
  \label{pr:val_by_wron}
  Suppose that $\pser_{1}(x), \dots, \pser_{n}(x)$ are Puiseux series with nonnegative valuations.
  Then, we have the inequality
\[
\val(\pser_{1}(x) + \pser_{2}(x) + \dots + \pser_{n}(x)) \le
    \frac{n(n - 1)}{2}+ \val\Bigl(\wron\bigl(\pser_{1}(x),\pser_{2}(x),\dots,\pser_{n}(x)\bigr)\Bigr) \, .
\]
\end{proposition}
\begin{proof}
  Let $\pserii(x) \coloneqq \pser_{1}(x) + \pser_{2}(x) + \dots + \pser_{n}(x)$.
    By multilinearity of the determinant we have
  $\wron\bigl(\pser_{1}(x),\pser_{2}(x),\dots,\pser_{n}(x)\bigr) = \wron\bigl(\pser_{1}(x),\dots,\pser_{n-1}(x), \pserii(x)\bigr)$.
  Using the Laplace expansion, we obtain
\[
\wron\bigl(\pser_{1}(x),\pser_{2}(x),\dots,\pser_{n-1}(x), \pserii(x)\bigr) = \sum_{k = 0}^{n-1}\diffp[k]{\pserii}{x}(x)M_{k} \, ,
\]
where $M_{k} \in \puiseux$ are some $(n-1)\times(n-1)$ minors of the matrix
in~\cref{eq:wron}. Since we assumed that
$\pser_{1}(x), \dots, \pser_{n}(x)$
have nonnegative valuations,~\cref{eq:val_dif} implies that
every entry in row $i$ of this matrix has valuation at least
  $-(i-1)$.
Hence, by~\cref{eq:val} we have 
\[
\val(M_{k}) \ge -\sum_{i=2}^{n} (i-1)+k=\frac{-n(n-1)}{2} + k \, .
\]
In particular,
\[
\val\Bigl( \wron\bigl(\pser_{1}(x),\pser_{2}(x),\dots,\pser_{n-1}(x), \pserii(x)\bigr) \Bigr) \ge \min_{k}\Bigl(\val\bigl(\diffp[k]{\pserii}{x}(x)\bigr) + \val(M_{k})\Bigr) \, .
\]
By~\cref{eq:val_dif}, the right-hand side is bounded from below by
\[
\min_{k}\bigl( \val(\pserii(x)) - k-\frac{n(n-1)}{2} + k\bigr)
= \val(\pserii(x))-\frac{n(n-1)}{2} \, . \qedhere
\]
\end{proof}

 \begin{remark} \label{re:val_by_wron}
  The original version of \cref{pr:val_by_wron}
  in~\cite{voorhoeve_wronskian} is about analytic functions rather than
  Puiseux series. The restriction to analytic functions makes it possible
  to obtain a better bound: instead of the term $n(n-1)/2$
  in  \cref{pr:val_by_wron} we have just $n-1$ in
  \cite[Theorem~1]{voorhoeve_wronskian}.
  \end{remark}

  \begin{example}
    Let $\pser_{1}(x)\coloneqq x^{\alpha_1},\dots,\pser_{n}(x) \coloneqq x^{\alpha_n}$
    where $0<\alpha_1 < \cdots < \alpha_n < 1$.
    The valuation of $\pser_{1}(x) + \pser_{2}(x) + \dots + \pser_{n}(x)$
    is equal to $\alpha_1$, and it is easily checked that
    \[
    \val\Bigl(\wron\bigl(\pser_{1}(x),\pser_{2}(x),\dots,\pser_{n}(x)\bigr)\Bigr) = \alpha_1+\cdots+\alpha_n - \frac{n(n-1)}{2} \, .
    \]
    Since the $\alpha_i$ can be taken as close to 0 as desired, this example
    shows that the inequality in \cref{pr:val_by_wron} is essentially optimal.
    \end{example}

\section{Intersection multiplicity} \label{sec:mult}

In this section, we recall the definition of intersection multiplicity of two curves and we give an equivalent characterization that is suitable for our purposes.

Let $\bratf$ be the field of rational functions in two variables over $\C$.
Then, for every $p = (a,b) \in \C^{2}$ we define the \emph{local ring at $p$}, $\lring_{p} \subset \bratf$, as the ring of all rational functions whose denominators do not vanish at $p$,
\[
\lring_{p} \coloneqq \Bigl\{\frac{F(x,y)}{G(x,y)} \colon F(x,y), G(x,y) \in \C[x,y], G(a,b) \neq 0 \Bigr\} \, .
\]

\begin{definition}
If $F(x,y), G(x,y) \in \C[x,y]$ are two polynomials and $p = (a,b) \in \C^{2}$ is any point, then we define the \emph{intersection multiplicity (of $F(x,y)$ and $G(x,y)$ at point $p$)} as
\[
\mint_{p}(F,G) \coloneqq \dim_{\C}\bigl( \lring_{p}/\langle F,G \rangle \bigr) \, ,
\]
where $\langle F,G \rangle$ is the ideal in $\lring_{p}$ generated by $F(x,y)$ and $G(x,y)$, and $\dim_{\C}$ refers to the dimension of $\lring_{p}/\langle F,G \rangle$ interpreted as a vector space over $\C$.
\end{definition}

The next lemma gathers some classical properties of intersection multiplicity.

\begin{lemma}\label{le:inter_mult}
  Intersection multiplicity has the following properties:
\begin{enumerate}
\item $\mint_{p}(F,G) = 0$ if and only if $F(a,b) \neq 0$ or $G(a,b) \neq 0$;
\item If $F(x,y)$ and $G(x,y)$ are nonzero polynomials, then $\mint_{p}(F,G) = +\infty$ if and only if $F(x,y)$ and $G(x,y)$ have a common 
  factor $H(x,y)$ that satisfies $H(a,b) = 0$;
\item $\mint_{p}(F, G) = \mint_{p}(G, F)$;
\item If $F(x,y) = F_{1}(x,y)F_{2}(x,y)$, then $\mint_{p}(F, G) = \mint_{p}(F_{1}, G) + \mint_{p}(F_{2},G)$;
\item If $L \colon \C^{2} \to \C^{2}$ is an invertible affine map and we define $\trF \coloneqq F \circ L$, $\trG \coloneqq G \circ L$, then $\mint_{p}(F,G) = \mint_{L^{-1}(p)}(\trF, \trG)$.
\end{enumerate}
\end{lemma}

There are many equivalent characterizations of intersection multiplicity. For instance, there is an axiomatic definition given in~\cite[Section~3.2]{fulton_curves}, a definition using resultants~\cite[Section~6.1]{brieskorn_knorrer},
a definition by parametrization~\cite[Chapter~I, Section~3.2]{greuel_lossen_shustin} or by infinitely near points~\cite[Section~4.4]{wall_singular_points}. In this work, we will use a variant of the characterization of the intersection multiplicity by parametrization. Suppose that $F(x,y), G(x,y) \in \C[x,y]$ are two polynomials. Since the field of Puiseux series is algebraically closed, we can decompose $F$ and $G$ as
\begin{equation}\label{eq:decomp_to_roots}
\begin{aligned}
F(x,y) &= \pser_{0}(x)(y - \pser_{1}(x)) \dots (y - \pser_{\degy_{F}}(x)) \, , \\
G(x,y) &= \pserii_{0}(x)(y - \pserii_{1}(x)) \dots (y - \pserii_{\degy_{G}}(x)) \, ,
\end{aligned}
\end{equation}
where
$\pser_{0}(x),\pserii_{0}(x) \in \C[X]$
and $\pser_{1}(x), \dots, \pser_{\degy_{F}}(x), \pserii_{1}(x), \dots, \pserii_{\degy_{G}}(x) \in \puiseux$ are Puiseux series (not necessarily distinct). Furthermore, we can order the factors in such a way that there are two numbers $0 \le r \le \degy_{F}$ and $0 \le s \le \degy_{G}$ such that the series $\pser_{1}(x), \dots, \pser_{r}(x)$, $\pserii_{1}(x), \dots, \pserii_{s}(x)$ have \emph{strictly} positive valuations,\footnote{In this list we include the series that are identically 0 since their valuation is $+\infty$ by convention.}
while the valuations of the series $\pser_{r + 1}(x), \dots, \pser_{\degy_{F}}(x)$ and $\pserii_{s + 1}(x), \dots, \pserii_{\degy_{G}}(x)$ are zero or smaller than zero. 
Moreover, let $m \ge 0$ be the highest number such that $F(x,y)$ is divisible by $x^{m}$ and $n \ge 0$ be the highest number such that $G(x,y)$ is divisible by $x^{n}$.

The following proposition characterizes the intersection multiplicity.

\begin{proposition}\label{le:halphen}
If $m > 0$ and $n > 0$, then $\mint_{0}(F,G) = +\infty$. Otherwise, we have the equality
\begin{equation}\label{eq:halphen}
\begin{aligned}
\mint_{0}(F,G) &= m s + \sum_{i =1}^{r}\val\Bigl( G\bigl(x,\pser_{i}(x)\bigr) \Bigr) \\
&= n r + \sum_{j =1}^{s}\val\Bigl( F\bigl(x,\pserii_{j}(x)\bigr) \Bigr) \\
&= m s + n r + \sum_{i = 1}^{r}\sum_{j = 1}^{s}\val\bigl(\pser_{i}(x) - \pserii_{j}(x)\bigr) \, .
\end{aligned}
\end{equation}
Furthermore, if $p = (a,b) \in \C^{2}$ is any point and we define $\trF(x,y) \coloneqq F(a + x, b + y)$, $\trG(x,y) \coloneqq G(a + x, b + y)$, then $\mint_{p}(F, G) = \mint_{0}(\trF, \trG)$.
\end{proposition}
We note that \cref{eq:halphen} is sometimes proven under additional assumptions (such as $m=n=0$), see, e.g.,~\cite[Chapter~4, Section~5.1]{walker_algebraic_curves} or~\cite[Section~4.1]{wall_singular_points}. However, the variant stated in \cref{le:halphen} is valid in general: for a detailed proof,
we refer to \cite[Chapter~IV]{bix_conics} and in particular to Definition~14.4 and Theorem~14.6 of this reference. The second part of \cref{le:halphen} follows from \cref{le:inter_mult}(5).

We finish this section by stating some known results. The first one states that the numbers $r,s$ can be easily characterized by means of Newton polygons. This follows from the Newton--Puiseux algorithm. Although the knowledge of this algorithm is not necessary to understand the results of this paper (we only need \cref{pr:val_positive_roots} stated below), it is useful to point out the main features of this algorithm. The Newton--Puiseux algorithm allows us to compute the decomposition given in~\cref{eq:decomp_to_roots}. To compute this decomposition, we denote 
\begin{align*}
F(x,y) &= F_{0}(x) + F_{1}(x)y + \dots + F_{\degy_{F}}(x)y^{\degy_{F}} \, , \\
G(x,y) &= G_{0}(x) + G_{1}(x)y + \dots + G_{\degy_{G}}(x)y^{\degy_{G}} \,
\end{align*}
and we note that $m = \min_{k}\{\val(F_{k}(x)) \}$,
$n = \min_{k}\{\val(G_{k}(x))\}$. Then, we define the \emph{Newton polygons} of $F$ and $G$ as the convex hulls of points 
\begin{align*}
&\{\bigl(k, \val(F_{k}(x))\bigr) \colon 0 \le k \le \degy_{F},  \, F_{k}(x) \neq 0 \} \, ,\\ 
&\{\bigl(k, \val(G_{k}(x))\bigr) \colon 0 \le k \le \degy_{G}, \, G_{k}(x) \neq 0 \} \, .
\end{align*}
\begin{remark}
  The Newton polygon is sometimes defined as the convex hull of the points
  $(\alpha,\beta)$ such that the monomial $y^{\alpha}x^{\beta}$ appears in
  $F$ with a nonzero coefficient. These two polygons have the same set
  of lower edges, and as explained below this is all that matters to
  determine the valuations of the series in~\cref{eq:decomp_to_roots}.
\end{remark}
The Newton--Puiseux algorithm implies that the valuations of the series $\pser_{1}(x), \dots, \pser_{\degy_{F}}(x), \pserii_{1}(x), \dots, \pserii_{\degy_{G}}(x) \in \puiseux$ are given by the (negated) slopes of the lower edges of the corresponding Newton polygons. 
Furthermore, the number of series with a given valuation (counted with multiplicity) is equal to the length of the projection of the corresponding edge on the first axis. This does not include the series that are equal to $0$, but their number can also be deduced from the Newton polygons, since it is equal to
 $\min\{k \colon F_{k} \neq 0\}$
  and $\min\{k \colon G_{k} \neq 0 \}$
  respectively. In particular, we obtain the following characterization of the numbers of series in  \cref{eq:decomp_to_roots}
with strictly positive valuations (denoted by $r$ and $s$
as in the paragraphs above).
It will be used in the proof of \cref{le:card_zeros_of_shift}.

\begin{proposition}\label{pr:val_positive_roots}
  We have the equalities $r = \min\{k \colon \val(F_{k}(x)) = m\}$ and $s = \min\{k \colon \val(G_{k}(x)) = n\}$.
\end{proposition}

\begin{figure}[t]
\begin{tikzpicture}
\begin{scope}[scale = 0.8]
      \draw[gray!60, ultra thin] (0,0) grid (8,4);
       \fill[fill=lightgray, fill opacity = 0.7]
        (1,3) -- (2,2) -- (3,1) -- (4,1) -- (7,2) -- (5,3) -- cycle;
        \draw[very thick] (1,3) -- (2,2) -- (3,1) -- (4,1) -- (7,2) -- (5,3) -- cycle;
        \node[circle,fill, inner sep = 0pt,minimum size = 0.2cm] at (1,3) {};
        \node[circle,fill, inner sep = 0pt,minimum size = 0.2cm] at (2,2) {};
        \node[circle,fill, inner sep = 0pt,minimum size = 0.2cm] at (3,1) {};
        \node[circle,fill, inner sep = 0pt,minimum size = 0.2cm] at (4,1) {};
        \node[circle,fill, inner sep = 0pt,minimum size = 0.2cm] at (5,3) {};
        \node[circle,fill, inner sep = 0pt,minimum size = 0.2cm] at (6,2) {};
        \node[circle,fill, inner sep = 0pt,minimum size = 0.2cm] at (7,2) {};
\end{scope}
\end{tikzpicture}
\vspace*{-0.3cm}
\caption{Newton polygon from \cref{ex:newton}.}\label{fig:newton}
\end{figure}

\begin{example}\label{ex:newton}
Consider the polynomial
\[
F(x,y) = xy(y - x + x^{2})^{2}(y - 1 + x)(xy^{3} - 1) \, .
\]
To find its decomposition, note that
\[
(xy^{3} - 1) = x(y^{3} - x^{-1}) = x(y - x^{1/3})(y - \omega x^{1/3})(y - \omega^{2} x^{1/3}) \, ,
\]
where $\omega \coloneqq (-1 + \imath \sqrt{3})/2$ is a third root of unity.
Therefore, we have
\[
F(x,y) = x^2\prod_{i = 1}^{7}(y - \pser_{i}(x)) \, ,
\]
where $\pser_{1}(x) = 0$, $\pser_{2}(x) = \pser_{3}(x) = x - x^{2}$, $\pser_{4}(x) = 1 - x$, $\pser_{5}(x) = x^{-1/3}$, $\pser_{6}(x) = \omega x^{-1/3}$, $\pser_{7}(x) = \omega^{2} x^{-1/3}$. Note that exactly three of these series have strictly positive valuation (namely $\pser_{1}(x)$, $\pser_{2}(x)$, and $\pser_{3}(x)$). Furthermore, we have
\begin{align*}
F(x,y) &= y(x^3-3x^4+3x^5-x^6) + y^{2}(-2x^2+3x^3-x^5) \\
&+ y^{3}(x+x^2-2x^3) + y^{4}(-x-x^4+3x^5-3x^6+x^7) \\
&+ y^{5}(2x^3-3x^4+x^6) + y^{6}(-x^2-x^3+2x^4) + y^{7}x^{2} \, .
\end{align*}
In particular, the Newton polygon of $F$ is the convex hull of the points 
\[
\{(1,3), (2,2), (3,1), (4,1), (5,3), (6,2), (7,2)\} \, ,
\]
as depicted in \cref{fig:newton}. Its lower edges have slopes $-1$, $0$, and $1/3$, while the lengths of their projections on the abscissa are equal to $2$, $1$, and $3$ respectively. As discussed above, the edge with slope $-1$ indicates that the decomposition has two series of valuation $1$ (these are
$\pser_{2}(x)$ and $\pser_{3}(x)$),
the edge with slope $0$ indicates that the decomposition has one series with valuation $0$ (this is $\pser_{3}(x)$), and the edge with slope $1/3$ indicates that the decomposition has three series with valuation $-1/3$ (these are
 $\pser_{5}(x)$, $\pser_{6}(x)$, and $\pser_{7}(x)$).
Furthermore, we have 
$\min\{k \colon F_{k} \neq 0\} =1$
and $\min\{k \colon \val(F_{k}(x)) = m\} = 3$, which, as claimed in \cref{pr:val_positive_roots}, is equal to the number of series with strictly positive valuation.
\end{example}

We refer to \cite[Chapter~1]{casas-alvero_singularities} for a detailed presentation of the Newton--Puiseux algorithm and in particular to \cite[Exercise~1.3]{casas-alvero_singularities} for the correspondence between the number of series with a given valuation and the length of the projection of the corresponding edge.

The next result follows from Gauss' lemma and the fact that irreducible polynomials over fields of characteristic zero are separable.

\begin{lemma}\label{le:gauss}
  Suppose that $F(x,y) \in \C[x,y]$ is an irreducible bivariate polynomial over $\C$ and consider the decomposition of $F(x,y)$ given in~\cref{eq:decomp_to_roots}. Then, the roots $\pser_{1}(x), \dots \pser_{\degy_{F}}(x) \in \puiseux$ are pairwise distinct.   Furthermore, if $G(x,y) \in \C[x,y]$ is any polynomial that satisfies $G(x,\pser_{i}(x)) = 0$ for some $1 \le i \le d$, then $F(x,y)$ divides $G(x,y)$  in $\C[x,y]$.
\end{lemma}

\begin{proof}
  Gauss' lemma (see, e.g.,~\cite[Section~2.6]{adkins_weintraub} or \cite[Section~4.2]{lang_algebra}) implies that $F$ is still irreducible if we consider it as an element of $\bigl(\C(x) \bigr)[y]$ (polynomials
with coefficients in the field of rational functions of $x$)
  since $\C[x]$ is a unique factorization domain.
    Therefore, the fact that the series $\pser_{1}(x), \dots \pser_{\degy_{F}}(x) \in \puiseux$ are pairwise distinct follows from the separability of irreducible polynomials in $\bigl(\C(x) \bigr)[y]$ (see, e.g.,~\cite[Proposition~4.6]{morandi_fields} or~\cite[Corollary~6.12]{lang_algebra}). To prove the second part, suppose that $G(x, \pser_{i}(x)) = 0$ for some $i$. Then, the polynomials $F$ and $G$ have a nontrivial greatest common divisor in $\bigl( \puiseux \bigr)[y]$. However, since both $F$ and $G$ belong to the subring $\bigl(\C(x) \bigr)[y]$, their greatest common divisor also belongs to this subring (because the gcd can be computed using Euclidean division). As $F$ is irreducible, we obtain that $G$ is divisible by $F$ in $\bigl(\C(x) \bigr)[y]$. We use Gauss' lemma once again to conclude that $G$ is divisible by $F$ in $\C[x,y]$.
\end{proof}

The next two results are B{\'e}zout's theorem and the Haj{\'o}s lemma.

\begin{theorem}[B{\'e}zout's theorem in affine space]\label{th:bezout}
Let $F(x,y),G(x,y) \in \C[x,y]$ be two polynomials of degrees $d_{1} \ge 0$ and $d_{2} \ge 0$ respectively. Let 
\[
\Omega \coloneqq \{(a,b) \in \C^{2} \colon F(a,b) = G(a,b) = 0, \, \mint_{(a,b)}(F,G) < +\infty \}
\]
be the set of isolated solutions of the system $F(x,y) = G(x,y) = 0$. Then, we have the inequality
\[
\sum_{p \in \Omega}\mint_{p}(F,G) \le d_{1}d_{2} \, .
\]
\end{theorem}
    The following result can be found in~\cite{hajos53} and in a
      more general form in~\cite[Proposition~3.2]{Lenstra99b}. We provide a short proof
      for the sake of completeness.
\begin{lemma}[Haj{\'o}s' lemma]\label{le:hajos}
Suppose that $F(y) \in \C[y]$ is a univariate polynomial with $t \ge 1$ monomials and let $z \in \C \setminus \{0\}$ be a nonzero root of $F(y)$. Then, the multiplicity of $z$ as root of $F(y)$ is not greater than $t - 1$. 
\end{lemma}
\begin{proof}
We prove the claim by induction on $t$. If $t = 1$, then the claim is trivial. Otherwise, we can write $F(y) = y^{m}(a_{0} + a_{1}y + \dots + a_{d}y^{d})$ for some $m,d \in \N$ and complex constants $(a_{0}, \dots, a_{d})$ such that $t$ of them are nonzero and $a_{0} \neq 0$. We note that is it enough to prove the claim for $G(y) \coloneqq a_{0} + a_{1}y + \dots + a_{d}y^{d}$. Let $z \in \C\setminus\{0\}$ be a nonzero root of $G(y)$. If the multiplicity of $z$ is higher than $1$, then $z$ is a root of $G'(y) = a_{1} + 2a_{2}y + \dots + da_{d}y^{d-1}$. Moreover, by the induction hypothesis, the multiplicity of $z$ as root of $G'(y)$ is not higher than $t - 2$. Hence, the multiplicity of $z$ as root of $G(y)$ is not higher than $t - 1$.
\end{proof}

\section{Two lemmas about derivatives}\label{sec:deriv_lemmas}

In this section, we present two lemmas about derivatives that are used in the proof of our main theorem. These results appeared in~\cite{koiran_sparse_curve,koiran_wronskian_approach} in the context of analytic functions and they carry over to Puiseux series. We use the convention that $\N = \{0,1,\dots\}$ and $\posN = \{1,2,\dots\}$. For every $k \ge 0$, we denote by $\powseq_{k} \subset \N^{\posN}$ the set of sequences defined as
\[
\powseq_{k} \coloneqq \bigl\{(s_{1},s_{2},\dots) \in \N^{\posN} \colon \sum_{i = 1}^{\infty}is_{i} = k \bigr\} \, .
\]
We note that every sequence in $\powseq_{k}$ has finitely many nonzero entries. Furthermore, for every $s \in \powseq_{k}$ we denote $\abs{s} \coloneqq \sum_{i = 1}^{\infty}s_{i}$ and we note that $\abs{s} \le k$. 

\begin{lemma}\label{le:deriv_power}
  There exist integer constants $(\powcon_{n, s})_{n \in \N, s \in \N^{\posN}}$ such that for every nonzero Puiseux series 
$\pser(x) \in \puiseux$  and every $k,n \in \N$ we have
\[
\diffp[k]{}{x}\bigl( \pser(x)^{n}\bigr) = \sum_{s \in \powseq_{k}} \powcon_{n,s} \pser(x)^{n - \abs{s}}\prod_{l = 1}^{k}\Bigl( \diffp[l]{\pser}{x}(x) \Bigr)^{s_{l}} \, .
\]
(We use the convention that $0^{0} = 1$ and that an empty product is equal to~$1$.)
\end{lemma}
The proof of~\cref{le:deriv_power} proceeds by induction 
on $k$, using the elementary properties of derivatives given in~\cref{eq:der}. We refer to~\cite[Lemma~10]{koiran_wronskian_approach} for the details. The following lemma appeared in~\cite[Lemma~3]{koiran_sparse_curve}.

\begin{lemma}\label{le:root_deriv}
  For every $k \in \posN$ there exists a polynomial $\rootpol_{k} \in \Z[x_{1}, \dots,x_{l}]$ with integer coefficients, $l \coloneqq \binom{k+2}{2} - 1 = \frac{1}{2}k(k+3)$ variables, and degree at most $2k - 1$ such that for every pair
  $F(x,y) \in \C[x,y] \setminus \{0\}$, $\pser(x) \in \puiseux$ that satisfies $F(x,\pser(x)) = 0$ we have
\[
\Bigl(\diffp[k]{\pser}{x}(x)\Bigr)\Bigl( \diffp{F}{y}(x,\pser(x))\Bigr)^{2k - 1} = \rootpol_{k}\biggl( \Bigl(\diffp[p,q]{F}{x,y}(x,\pser(x))\Bigr)_{1 \le p+q \le k}\biggr) \, .
\]
\end{lemma}
    For instance, the case $k=1$ of this lemma is:
      \[
      S'(x) \diffp{F}{y}(x,\pser(x)) =
      -\diffp{F}{x}(x,\pser(x)) \, .
      \]
    The proof presented in \cite{koiran_sparse_curve} is based on the fact that
        for any polynomial $P \in \C[x_{1}, \dots, x_{n}]$ 
    and any Puiseux series $\pser_{1}(x), \dots, \pser_{n}(x) \in \puiseux$
     we have
\[
\diffp[]{}{x}\Bigl( P\bigl(\pser_{1}(x), \dots, \pser_{n}(x)\bigr) \Bigr) = \sum_{k = 1}^{n} \diffp{P}{x_{k}}\bigl(\pser_{1}(x), \dots, \pser_{n}(x)\bigr)\diffp{\pser_k}{x}(x) \, .
\]
The proof in \cite{koiran_sparse_curve} is in fact stated for analytic
functions, but it is easily checked that the same proof applies
to Puiseux series. An alternative proof of \cref{le:root_deriv}
can be found in the appendix to the present paper.

\section{Proof of the main theorem} \label{sec:proof}

In this section, we give the proof of \cref{th:main}. The proof is based on the following two lemmas. The first one relies on the Haj{\'o}s lemma.

\begin{lemma}\label{le:card_zeros_of_shift}
  Suppose that $G(x,y) \in \C[x,y]$ is a polynomial with $t \ge 1$ monomials. Furthermore, fix $(a,b) \in (\C \setminus \{0\})^{2}$ and let 
  \[
  \trG(x,y) \coloneqq G(a + x, b+ y).
  \]
  Then, $\trG(x,y) \in (\puiseux)[y]$ regarded as a polynomial 
  in variable $y$ with coefficients in the field of Puiseux series
  has at most $t-1$ roots (counted with multiplicity)
  that have strictly positive valuations.
\end{lemma}

\begin{proof}
Let $G(x,y) = \sum_{\alpha \in \psupport}c_{\alpha}x^{\alpha_{1}}y^{\alpha_{2}}$ with $c_{\alpha} \neq 0$ for all $\alpha \in \psupport$. Let $m_{i} = \max\{\alpha_{i} \colon (\alpha_{1},\alpha_{2}) \in \psupport\}$ for $i \in \{1,2\}$ and
\[
\trG_{k,l} \coloneqq \sum_{\alpha \in \psupport, \alpha_{1} \ge k, \alpha_{2} \ge l}\binom{\alpha_{1}}{k}\binom{\alpha_{2}}{l} c_{\alpha}a^{\alpha_{1} - k}b^{\alpha_{2} - l} \,
\]
for every $0 \le k \le m_{1}$, $0 \le l \le m_{2}$.
Then, we obtain
\[
\trG(x,y) = \sum_{\alpha \in \psupport}c_{\alpha}(a + x)^{\alpha_{1}}(b +y)^{\alpha_{2}} = \sum_{k = 0}^{m_{1}}\sum_{l = 0}^{m_{2}}\trG_{k,l}x^{k}y^{l} \, .
\]
For every $0 \le l \le m_{2}$, let $\trG_{l}(x) = \sum_{k = 0}^{m_{1}}\trG_{k,l}x^{k} \in \puiseux$. Let $n$ be the highest number such that $x^{n}$ divides $\trG(x,y)$, i.e., $n \coloneqq \min_{l}\{\val(\trG_{l}(x))\}$.
Let $s$ be the number of roots of $\trG(x,y) \in (\puiseux)[y]$ with strictly positive valuation, counted with their multiplicities.
By \cref{pr:val_positive_roots}, $s =  \min\{l \colon \val(\trG_{l}(x)) = n\}$.
In particular, $s$ is the smallest number such that $\trG_{n,s} \neq 0$.
Consider the univariate polynomial
\[
H(y) \coloneqq \sum_{\alpha \in \psupport, \alpha_{1} \ge n}\binom{\alpha_{1}}{n} c_{\alpha}a^{\alpha_{1} - n}y^{\alpha_{2}}\, .
\]
Denote $\trH(y) \coloneqq H(b + y)$ and observe that
\[
\trH(y) = \sum_{l = 0}^{m_{2}}\trG_{n,l}y^{l} \, .
\]
Therefore, $s$ is equal to the multiplicity of $b$ as root of $H(y)$ (and $s = 0$ if $H(b) \neq 0$). Hence, by~\cref{le:hajos}, we have $s \le t - 1$.
\end{proof}

\begin{lemma}\label{le:sum_val_irred}
Suppose that $F(x,y) \in \C[x,y]$ is an irreducible polynomial of degree $d \ge 1$ and $G(x,y) \in \C[x,y]$ is a polynomial with $t \ge 1$ monomials that is not divisible by $F(x,y)$. Furthermore, fix $(a,b) \in (\C \setminus \{0\})^{2}$, and let $\pser_{1}(x), \dots, \pser_{r}(x) \in \puiseux$ denote all the series with strictly positive valuations such that $F(a + x, b + \pser_{i}(x)) = 0$ for $1 \le i \le r$.
 Then, we have 
\[
\sum_{i = 1}^{r}\val\Bigl(G\bigl(a + x,b + \pser_{i}(x)\bigr)\Bigr) \le \frac{1}{2}d(4d+1)t(t-1) \, .
\]
\end{lemma}
\begin{proof}
  We proceed by induction   on $t$. If $t = 1$, then $\val\Bigl(G\bigl(a + x,b + \pser_{i}(x)\bigr)\Bigr) = 0$ for all $i$ and the claim holds. Otherwise, denote $G(x,y) = \sum_{\alpha \in \psupport}c_{\alpha}x^{\alpha_{1}}y^{\alpha_{2}}$ with $c_{\alpha} \neq 0$ for all $\alpha \in \psupport$ and $\card{\psupport} = t$. Furthermore, denote the elements of $\psupport$ by $\psupport = \{\alpha^{(1)}, \alpha^{(2)}, \dots, \alpha^{(t)}\}$.
\begin{enumerate}[label= \textit {Case \Roman*}:, wide]
\item Suppose that $\wron\Bigl( \bigl( (a + x)^{\alpha^{(k)}_{1}}(b + \pser_{i}(x))^{\alpha^{(k)}_{2}}\bigr)_{k = 1}^{t} \Bigr) = 0$ for some~$i$. Then, by \cref{th:bocher} there exists a nonzero polynomial
  \[
  H(x,y) = \sum_{\alpha \in \psupport}\tilde{c}_{\alpha}x^{\alpha_{1}}y^{\alpha_{2}} \in \C[x,y]
  \]
such that $H(a + x,b + \pser_{i}(x)) = 0$. Hence, by \cref{le:gauss}, $H(x,y)$ is divisible by $F(x,y)$. In particular, the equality $H(a + x,b + \pser_{i}(x)) = 0$ holds for all~$i$. Let $\alpha^{*}$ be such that $\tilde{c}_{\alpha^{*}} \neq 0$. Then, the polynomial
\[
\hat{G}(x,y) \coloneqq G(x,y) - \frac{c_{\alpha^{*}}}{\tilde{c}_{\alpha^{*}}}H(x,y)
\]
has at most $t - 1$ monomials and satisfies
\[
G\bigl(a + x, b + \pser_{i}(x)\bigr) = \hat{G}\bigl(a + x, b + \pser_{i}(x)\bigr)
\]
for all $i$. Moreover, $\hat{G}(x,y)$ is not divisible by $F(x,y)$ because $G(x,y)$ is not divisible by $F(x,y)$. In particular, $\hat{G}(x,y)$ is a nonzero polynomial. Therefore, the claim follows by applying the induction hypothesis to $\hat{G}(x,y)$.
\item Suppose that $\wron\Bigl( \bigl( (a + x)^{\alpha^{(k)}_{1}}(b + \pser_{i}(x))^{\alpha^{(k)}_{2}}\bigr)_{k = 1}^{t} \Bigr) \neq 0$ for all $i$. By \cref{pr:val_by_wron}, it is enough to bound the valuation of this Wronskian in order to bound the sum $\sum_{i = 1}^{r}\val\bigl(G\bigl(a + x,b + \pser_{i}(x)\bigr)\bigr)$. To do so, let $\pser(x) \in \puiseux$ be any of the series $\pser_{1}(x), \dots, \pser_{r}(x)$ and denote by $\sym(t)$ the group of permutations of $\{0, \dots, t-1\}$. We have
\begin{equation}\label{eq:wron_by_perm}
\begin{aligned}
&\wron\Bigl( \bigl( (a + x)^{\alpha^{(k)}_{1}}(b + \pser(x))^{\alpha^{(k)}_{2}}\bigr)_{k = 1}^{t} \Bigr)\\
  &= \sum_{\sigma \in \sym(t)} \sign(\sigma) 
  \prod_{k = 0}^{t - 1}\Bigl( \diffp[k]{}{x}\bigl( (a + x)^{\alpha^{(\sigma(k))}_{1}} (b + \pser(x))^{\alpha^{(\sigma(k))}_{2}}\bigr)\Bigr) \, .
\end{aligned}
\end{equation}
Moreover, by \cref{le:deriv_power}, for every $\alpha \in \psupport$ and every $0 \le k \le t - 1$ we have
\begin{equation}\label{eq:wron_entry}
\begin{aligned}
&\diffp[k]{}{x}\bigl( (a + x)^{\alpha_{1}} (b + \pser(x))^{\alpha_{2}}\bigr)  \\
&= \sum_{l = 0}^{k}\binom{k}{l} \Bigl( \diffp[l]{}{x}(a+x)^{\alpha_{1}}\Bigr)\Bigl( \diffp[k-l]{}{x}(b+\pser(x))^{\alpha_{2}}\Bigr) \\
&= \sum_{l = 0}^{p}\sum_{\seqit \in \powseq_{k-l}}\biggl( \binom{k}{l}\frac{\alpha_{1}!}{(\alpha_{1} - l)!}\powcon_{\alpha_{2},\seqit}(a +x)^{\alpha_{1} - l} (b + \pser(x))^{\alpha_{2} - \abs{\seqit}}\prod_{j = 1}^{k-l}\Bigl( \diffp[j]{\pser}{x}(x) \Bigr)^{\seqit_{j}} \biggr) \, ,
\end{aligned}
\end{equation}
where $p \coloneqq \min\{k, \alpha_{1}\}$. Let $\trF(x,y) \coloneqq F(a + x, b + y) \in \C[x,y]$. Since $F(x,y)$ is irreducible, $\trF(x,y)$ is also irreducible, and \cref{le:gauss} shows that $\pser(x)$ is a root of $\trF(x,y) \in (\puiseux)[y]$ of multiplicity $1$. In particular, we have $\diffp{\trF}{y}(x,\pser(x)) \neq 0$. Hence, by \cref{le:root_deriv}, for every $1 \le j \le k$ there exists a polynomial $\rootpol_{j} \in \Z[x_{1}, \dots, x_{j(j+3)/2}]$ of degree at most $2j - 1$ such that
\begin{equation}\label{eq:root_main}
\diffp[j]{\pser}{x}(x) = \bigl( \diffp{\trF}{y}(x,\pser(x))\bigr)^{1 - 2j}\rootpol_{j}\Bigl( \bigl(\diffp[p,q]{\trF}{x,y}(x,\pser(x))\bigr)_{1 \le p+q \le j}\Bigr) \, .
\end{equation}
We note that $\sum_{j = 1}^{k - l}(2j - 1)\seqit_{j} = 2k - 2l - \abs{\seqit}$
for every $\seqit \in \powseq_{k-l}$.
In particular, for every $\seqit \in \powseq_{k-l}$ we have
\begin{equation}\label{eq:prod_of_deriv}
\begin{aligned}
&\prod_{j = 1}^{k-l}\Bigl( \diffp[j]{\pser}{x}(x) \Bigr)^{\seqit_{j}} \\ &= \Bigl( \diffp{\trF}{y}(x,\pser(x))\Bigr)^{\abs{\seqit} + 2l - 2k}\prod_{j = 1}^{k-l}\biggl(\rootpol_{j}\Bigl( \bigl(\diffp[p,q]{\trF}{x,y}(x,\pser(x))\bigr)_{1 \le p+q \le j}\Bigr) \biggr)^{\seqit_{j}} \, .
\end{aligned}
\end{equation}
We now want to combine~\cref{eq:wron_entry} and~\cref{eq:prod_of_deriv}. To do so, fix $0 \le k \le t - 1$ and note that for every $0 \le l \le p$ and every $\seqit \in \powseq_{k-l}$ we have
\begin{equation}\label{eq:exponents_in_product}
\begin{aligned}
&(a +x)^{\alpha_{1} - l} (b + \pser(x))^{\alpha_{2} - \abs{\seqit}}\Bigl( \diffp{\trF}{y}(x,\pser(x))\Bigr)^{\abs{\seqit} + 2l - 2k} \\
&= \frac{(a + x)^{\alpha_{1} - k}(b + \pser(x))^{\alpha_{2} - k}}{\bigl( \diffp{\trF}{y}(x,\pser(x))\bigr)^{2k}} (a + x)^{k - l}(b + \pser(x))^{k - \abs{\seqit}}\Bigl( \diffp{\trF}{y}(x,\pser(x))\Bigr)^{\abs{\seqit} + 2l} \, .
\end{aligned}
\end{equation}
Furthermore, by \cref{le:root_deriv}, the product $\prod_{j = 1}^{k-l}R_{j}^{\seqit_{j}}$ has degree at most $\sum_{j = 1}^{k - l}(2j - 1)\seqit_{j} = 2k - 2l - \abs{\seqit}$. Hence, by \cref{eq:wron_entry,eq:prod_of_deriv,eq:exponents_in_product}, we can write $\diffp[k]{}{x}\bigl( (a + x)^{\alpha_{1}} (b + \pser(x))^{\alpha_{2}}\bigr)$ as
\begin{align}\label{eq:wron_entry_poly}
\frac{(a + x)^{\alpha_{1} - k}(b + \pser(x))^{\alpha_{2} - k}}{\bigl( \diffp{\trF}{y}(x,\pser(x))\bigr)^{2k}}\wentpol_{\alpha, k}\Bigl( a +x, b + \pser(x),  \bigl(\diffp[p,q]{\trF}{x,y}(x,\pser(x))\bigr)_{1 \le p+q \le k} \Bigr) \, ,
\end{align}
where $\wentpol_{\alpha, k}$ is a polynomial with integer coefficients and degree not greater than
\begin{align*}
&\max_{0 \le l \le k, \seqit \in \powseq_{k - l}}\bigl( (k - l) + (k - \abs{\seqit}) + (\abs{\seqit} + 2l) + (2k - 2l - \abs{\seqit}) \bigr) \\
&= \max_{0 \le l \le k, \seqit \in \powseq_{k - l}}\bigl(4k - l - \abs{\seqit}) \le 4k \, .
\end{align*}
As a consequence of~\cref{eq:wron_by_perm,eq:wron_entry_poly}, we have
\begin{align*}
&\wron\Bigl( \bigl( (a + x)^{\alpha^{(k)}_{1}}(b + \pser(x))^{\alpha^{(k)}_{2}}\bigr)_{k = 1}^{t} \Bigr)\\
&=  \frac{(a + x)^{A_{1}}(b + \pser(x))^{A_{2}}}{\bigl( \diffp{\trF}{y}(x,\pser(x))\bigr)^{t(t-1)}}\wrpol_{\psupport}\Bigl( a +x, b + \pser(x),  \bigl(\diffp[p,q]{\trF}{x,y}(x,\pser(x))\bigr)_{1 \le p+q \le t-1} \Bigr) \, ,
\end{align*}
where $A_{i} \coloneqq \bigl( \sum_{k = 1}^{t} \alpha^{(k)}_{i}\bigr) - \binom{t}{2}$ for $i \in \{1,2\}$ and $\wrpol_{\psupport}$ is a polynomial with integer coefficients and degree not greater than $2t(t-1)$. Moreover, for all $1 \le p + q \le t - 1$ we have
\[
\diffp[p,q]{\trF}{x,y}(x,y) = \diffp[p,q]{F}{x,y}(a + x,b + y) \, .
\]
Hence, there exists a bivariate polynomial $\overbar{\wrpol}_{\psupport,F}(x,y) \in \C[x,y]$ of degree at most $2dt(t-1)$ such that
\begin{align*}
&\wron\Bigl( \bigl( (a + x)^{\alpha^{(k)}_{1}}(b + \pser(x))^{\alpha^{(k)}_{2}}\bigr)_{k = 1}^{t} \Bigr) \\
&= \frac{(a + x)^{A_{1}}(b + \pser(x))^{A_{2}}}{\bigl( \diffp{F}{y}(a + x,b + \pser(x))\bigr)^{t(t-1)}}\overbar{\wrpol}_{\psupport,F}\bigl( a +x, b + \pser(x) \bigr) \, .
\end{align*}
Furthermore, we have $0 \le \val\Bigl(\diffp{F}{y}(a + x,b + \pser(x))\bigr)\Bigr) < +\infty$ and, since $a,b \neq 0$, $\val\bigl( (a + x)^{A_{1}}(b + \pser(x))^{A_{2}} \bigr) = 0$. Moreover, since we assumed that the Wronskian $\wron\Bigl( \bigl( (a + x)^{\alpha^{(k)}_{1}}(b + \pser(x))^{\alpha^{(k)}_{2}}\bigr)_{k = 1}^{t} \Bigr)$ is not equal to $0$, we obtain $\val\Bigl( \overbar{\wrpol}_{\psupport,F}\bigl( a +x, b + \pser(x) \bigr) \Bigl) < + \infty$. In particular, we have
\begin{equation}\label{eq:val_wron}
\begin{aligned}
\val \biggl( \wron\Bigl( \bigl( &(a + x)^{\alpha^{(k)}_{1}}(b + \pser(x))^{\alpha^{(k)}_{2}}\bigr)_{k = 1}^{t} \Bigr) \biggr) \\
&\le \val \bigl( \overbar{\wrpol}_{\psupport,F}\bigl( a +x, b + \pser(x) \bigr) \bigr) < +\infty \, .
\end{aligned}
\end{equation}

We recall that the polynomials $\rootpol_{j}$ from \cref{le:root_deriv} do not depend on the choice of $\pser(x)$. Hence, the polynomials $\wentpol_{\alpha, k}$ and $\wrpol_{\psupport}$ that appear in the computations above also do not depend on $\pser(x)$. This implies that $\overbar{\wrpol}_{\psupport,F}(x,y)$ does not depend on the choice of $\pser(x)$. Hence, by \cref{pr:val_by_wron,eq:val_wron},
$\displaystyle \sum_{i = 1}^{r}\val\Bigl(G\bigl(a + x,b + \pser_{i}(x)\bigr)\Bigr)$
is upper bounded by
\begin{equation} \label{eq:sum_of_val}
 \frac{rt(t-1)}{2} + \sum_{i = 1}^{r}\val\Bigl( \overbar{\wrpol}_{\psupport,F}\bigl( a +x, b + \pser_{i}(x) \bigr) \Bigl) \, .
\end{equation}
Consider the system of equations
$F(x,y) = \overbar{\wrpol}_{\psupport,F}(x,y) = 0$ 
and assume that $\mint_{(a,b)}(F, \overbar{\wrpol}_{\psupport,F}) = +\infty$.
Since $F$ is irreducible, by \cref{le:inter_mult}(2) 
we obtain that $F(x,y)$ divides $\overbar{\wrpol}_{\psupport,F}(x,y)$.
Consequently we have 
$$\overbar{\wrpol}_{\psupport,F}\bigl( a +x, b + \pser_i(x) \bigr) = 0$$
  for every $i$,
which gives a contradiction with $\val\Bigl( \overbar{\wrpol}_{\psupport,F}\bigl( a +x, b + \pser(x) \bigr) \Bigl) < + \infty$. Therefore $\mint_{(a,b)}(F, \overbar{\wrpol}_{\psupport,F}) < +\infty$. Since the degree of $\overbar{\wrpol}_{\psupport,F}$ is not greater than $2dt(t-1)$, B{\'e}zout's theorem (\cref{th:bezout}) gives $\mint_{(a,b)}(F, \overbar{\wrpol}_{\psupport,F}) \le 2d^{2}t(t-1)$. By \cref{le:halphen} we get $\sum_{i = 1}^{r}\val\Bigl( \overbar{\wrpol}_{\psupport,F}\bigl( a +x, b + \pser_{i}(x) \bigr) \Bigl) \le 2d^{2}t(t-1)$. Since $r \le d$ we obtain
\[
\sum_{i = 1}^{r}\val\Bigl(G\bigl(a + x,b + \pser_{i}(x)\bigr)\Bigr) \le \frac{dt(t-1)}{2} + 2d^{2}t(t-1) = \frac{1}{2}d(4d+1)t(t-1) 
\]
from the upper bound~(\ref{eq:sum_of_val}). \qedhere
\end{enumerate}
\end{proof}

We are now ready to present the proof of our main theorem.

\begin{proof}[Proof of \cref{th:main}]
Let $F(x,y) \in \C[x,y]$ be a polynomial of degree $d \ge 1$ and $G(x,y) \in \C[x,y]$ be a polynomial with $t \ge 1$ monomials. Furthermore, suppose that $p \coloneqq (a,b) \in (\C \setminus \{0\})^{2}$ is a point such that $0 < \mint_{p}(F, G) < +\infty$. Factorize $F(x,y)$ as $F(x,y) = \prod_{k = 1}^{\ell}F_{k}(x,y)^{w_{k}}$, where $F_{k}(x,y) \in \C[x,y]$ are irreducible polynomials and let $d_{k} \ge 1$ denote the degree of $F_{k}(x,y)$. By \cref{le:inter_mult}(4) we have
\[
\mint_{p}(F,G) = \sum_{k = 1}^{\ell}w_{k}\mint_{p}(F_{k},G) \, .
\]
Take any $k$ such that $\mint_{p}(F_{k},G) \neq 0$. Since $\mint_{p}(F,G) < +\infty$, we have $\mint_{p}(F_{k},G) < +\infty$, and, by \cref{le:inter_mult}(2), $G(x,y)$ is not divisible by $F_{k}(x,y)$. We can now estimate $\mint_{p}(F_{k},G)$ using our previous results. To do so, let $\pser_{k,0}(x), \dots, \pser_{k, r_{k}}(x)$ denote all Puiseux series with strictly positive valuations such that $F_{k}(a + x, b + \pser_{k,i}(x)) = 0$. By \cref{le:halphen} we have
\[
\mint_{p}(F_{k},G) = m_{k}s + \sum_{i = 1}^{r_{k}}\val\Bigl(G\bigl(a + x,b + \pser_{k,i}(x)\bigr)\Bigr) \, ,
\]
where $m_{k}$ is the highest number such that
$F_{k}(a + x, b + y)$
is divisible by $x^{m_{k}}$ and $s$ is the number of series with strictly positive valuations in the decomposition of $G(a + x, b + y)$. In particular, we have $m_{k} \le d_{k}$. Furthermore, \cref{le:card_zeros_of_shift} shows that $s \le t - 1$ and \cref{le:sum_val_irred} shows that $\sum_{i = 1}^{r_{k}}\val\Bigl(G\bigl(a + x,b + \pser_{k,i}(x)\bigr)\Bigr) \le \frac{1}{2}d_{k}(4d_{k} + 1)t(t-1)$. Hence, we have
\[
\mint_{p}(F_{k},G) \le d_{k}(t-1) + \frac{1}{2}d_{k}(4d_{k} + 1)t(t-1) \le d_{k}(t-1) + \frac{1}{2}d_{k}(4d + 1)t(t-1)\, .
\]
As a consequence we obtain $\mint_{p}(F,G) \le d(t - 1) + \frac{1}{2}d(4d+1)t(t-1) = 2d^{2}t^{2} - 2d^{2}t + \frac{1}{2}dt^{2} + \frac{1}{2}dt - d$. Since $\frac{1}{2}dt - 2d^{2}t - d < 0$ and $\frac{1}{2}dt^{2} \le \frac{1}{2}d^{2}t^{2}$, we have $\mint_{p}(F,G) < \frac{5}{2}d^{2}t^{2}$.
\end{proof}

\section*{Acknowledgements}

The comments of the referees led to several improvements in the presentation of the paper.

\bibliographystyle{alpha}

\begin{thebibliography}{KPTT15}

\bibitem[AW92]{adkins_weintraub}
W.~A. Adkins and S.~H. Weintraub.
\newblock {\em Algebra. An Approach via Module Theory}, volume 136 of {\em
  Grad. Texts in Math.}
\newblock Springer, New York, 1992.

\bibitem[BB18]{briquel_burgisser_average}
I.~Briquel and P.~B{\"u}rgisser.
\newblock The real tau-conjecture is true on average.
\newblock \arxiv{1806.00417}, 2018.

\bibitem[BC76]{BC76}
A.~Borodin and S.~Cook.
\newblock On the number additions to compute specific polynomials.
\newblock {\em {SIAM} Journal on Computing}, 5(1):146--157, 1976.

\bibitem[BD10]{BoDu}
Alin Bostan and Philippe Dumas.
\newblock Wronskians and linear independence.
\newblock {\em The American Mathematical Monthly}, 117(8):722--727, 2010.

\bibitem[Bix06]{bix_conics}
R.~Bix.
\newblock {\em Conics and Cubics. A Concrete Introduction to Algebraic Curves}.
\newblock Undergrad. Texts Math. Springer, New York, second edition, 2006.

\bibitem[BK12]{brieskorn_knorrer}
E.~Brieskorn and H.~Kn{\"o}rrer.
\newblock {\em Plane Algebraic Curves. Translated by John Stillwell}.
\newblock Mod. Birkh{\"a}user Class. Birkh{\"a}user, Basel, 2012.

\bibitem[BN15]{binyamini15}
Gal Binyamini and Dmitry Novikov.
\newblock Multiplicities of {Noetherian} deformations.
\newblock {\em Geometric and Functional Analysis}, 25(5):1413--1439, 2015.

\bibitem[B{\^o}c00]{bocher_dependence}
M.~B{\^o}cher.
\newblock The theory of linear dependence.
\newblock {\em Ann. of Math.}, 2(1):81--96, 1900.

\bibitem[BS07]{BiSot}
F.~Bihan and F.~Sottile.
\newblock New fewnomial upper bounds from {Gale} dual polynomial systems.
\newblock {\em Moscow Mathematical Journal}, 7(3), 2007.

\bibitem[B{\"u}r00]{burgisser_completeness}
P.~B{\"u}rgisser.
\newblock {\em Completeness and Reduction in Algebraic Complexity Theory}.
\newblock Number~7 in Algorithms Comput. Math. Springer, Berlin, 2000.

\bibitem[CA00]{casas-alvero_singularities}
E.~Casas-Alvero.
\newblock {\em Singularities of Plane Curves}, volume 276 of {\em London Math.
  Soc. Lecture Note Ser.}
\newblock Cambridge University Press, Cambridge, 2000.

\bibitem[Ful69]{fulton_curves}
W.~Fulton.
\newblock {\em Algebraic Curves. An Introduction to Algebraic Geometry}.
\newblock W. A. Benjamin, Inc., New York, 1969.

\bibitem[Gab95]{gabrielov_multiplicity}
A.~Gabrielov.
\newblock Multiplicities of {P}faffian intersections, and the {{\L}ojasiewicz}
  inequality.
\newblock {\em Selecta Math.}, 1(1):113--127, 1995.

\bibitem[GK98]{gab98}
Andrei Gabrielov and Askold Khovanskii.
\newblock Multiplicity of a {Noetherian} intersection.
\newblock {\em Translations of the American Mathematical Society-Series 2},
  186:119--130, 1998.

\bibitem[GKPS11]{grenet_koiran_powering}
B.~Grenet, P.~Koiran, N.~Portier, and Y.~Strozecki.
\newblock The limited power of powering: polynomial identity testing and a
  depth-four lower bound for the permanent.
\newblock In {\em Proceedings of the 31st IARCS Annual Conference on
  Foundations of Software Technology and Theoretical Computer Science
  (FSTTCS)}, pages 127--139, 2011.

\bibitem[GLS07]{greuel_lossen_shustin}
G.-M. Greuel, C.~Lossen, and E.~Shustin.
\newblock {\em Introduction to Singularities and Deformations}.
\newblock Springer Monogr. Math. Springer, Berlin, 2007.

\bibitem[Gri82]{Gri82}
D.~Grigoriev.
\newblock Notes of the scientific seminars of {LOMI}.
\newblock Technical report, 1982.

\bibitem[Haj53]{hajos53}
Gy{\"o}rgy Haj{\'o}s.
\newblock Solution to problem 41 (in {Hungarian}).
\newblock {\em Mat. Lapok}, 4:40--41, 1953.

\bibitem[Hru13]{hrubes_complex_tau}
P.~Hrube\v{s}.
\newblock On the real $\tau$-conjecture and the distribution of complex roots.
\newblock {\em Theory Comput.}, 9:403--411, 2013.

\bibitem[Hru19]{hrubes_runners}
P.~Hrube\v{s}.
\newblock On the distribution of runners on a circle.
\newblock  \href{https://arxiv.org/pdf/1906.02511.pdf}{arXiv:1906.02511},
  2019.

\bibitem[Kho91]{Khov91}
A.~G. Khovanskii.
\newblock {\em Fewnomials}, volume~88 of {\em Translations of Mathematical
  Monographs}.
\newblock American Mathematical Society, 1991.

\bibitem[Koi11]{koiran_real_tau}
P.~Koiran.
\newblock Shallow circuits with high-powered inputs.
\newblock In {\em Proceedings of the Second Symposium on Innovations in
  Computer Science (ICS)}, pages 309--320, 2011.

\bibitem[KPT15a]{koiran_sparse_curve}
P.~Koiran, N.~Portier, and S.~Tavenas.
\newblock On the intersection of a sparse curve and a low-degree curve: {A}
  polynomial version of the lost theorem.
\newblock {\em Discrete Comput. Geom.}, 53(1):48--63, 2015.

\bibitem[KPT15b]{koiran_wronskian_approach}
P.~Koiran, N.~Portier, and S.~Tavenas.
\newblock A {W}ronskian approach to the real $\tau$-conjecture.
\newblock {\em J. Symbolic Comput.}, 68(2):195--214, 2015.

\bibitem[KPTT15]{koiran_tau_newton}
P.~Koiran, N.~Portier, S.~Tavenas, and S.~Thomass{\'e}.
\newblock A $\tau$-conjecture for {N}ewton polygons.
\newblock {\em Found. Comput. Math.}, 15(1):185--197, 2015.

\bibitem[Lang93]{lang_algebra}
  S. Lang.
  \newblock Algebra. \newblock Addison-Wesley, third edition, 1993.

\bibitem[Len99]{Lenstra99b}
H.~W. Lenstra.
\newblock Finding small degree factors of lacunary polynomials.
\newblock In {\em Number theory in progress, Vol.~1}, pages 267--276. De
  Gruyter, Berlin, 1999.
  
\bibitem[Mor96]{morandi_fields}
P.~Morandi.
\newblock {\em Field and {G}alois Theory}, volume 167 of {\em Grad. Texts in
  Math.}
\newblock Springer, New York, 1996.

\bibitem[Ris85]{Risler85}
J.-J. Risler.
\newblock Additive complexity and zeros of real polynomials.
\newblock {\em {SIAM} Journal on Computing}, 14:178--183, 1985.

\bibitem[Sma98]{SmaleProblems}
S.~Smale.
\newblock Mathematical problems for the next century.
\newblock {\em Mathematical Intelligencer}, 20(2):7--15, 1998.

\bibitem[Sot11]{sottile11}
Frank Sottile.
\newblock {\em Real solutions to equations from geometry}, volume~57.
\newblock American Mathematical Society, 2011.

\bibitem[Tav14]{tavenasphd}
S.~Tavenas.
\newblock {\em Bornes inf\'erieures et sup\'erieures dans les circuits
  arithmétiques.}
\newblock PhD thesis, Ecole Normale Sup\'erieure de Lyon, 2014.

\bibitem[VvdP75]{voorhoeve_wronskian}
M.~Voorhoeve and A.~J. van~der Poorten.
\newblock Wronskian determinants and the zeros of certain functions.
\newblock {\em Indag. Math.}, 78(5):417--424, 1975.

\bibitem[Wal78]{walker_algebraic_curves}
R.~J. Walker.
\newblock {\em Algebraic Curves}.
\newblock Springer, New York, 1978.

\bibitem[Wal04]{wall_singular_points}
C.~T.~C. Wall.
\newblock {\em Singular Points of Plane Curves}, volume~63 of {\em London Math.
  Soc. Stud. Texts}.
\newblock Cambridge University Press, Cambridge, 2004.

\end{thebibliography}

\appendix

\section{Additional proofs}\label{app:additional_proofs}

In this Appendix, we give a proof of \cref{le:root_deriv}. This result is true not only for polynomials $F(x,y) \in \C[x,y] \setminus\{0\}$ but also for polynomials over Puiseux series, $F(x,y) \in (\puiseux)[y] \setminus \{0\}$. To prove the lemma in this context, we need to extend the definition of derivation from Puiseux series to polynomials over Puiseux series. This is done in the natural way.

\begin{definition}
If $F(x,y) \in (\puiseux)[y]$, $F(x,y) = \sum_{k = 0}^{d}F_{k}(x)y^{k}$ is a polynomial over Puiseux series, then we define its \emph{derivative with respect to $x$} as
\[
\diffp{F}{x}(x,y) \coloneqq \sum_{k = 0}^{d}\diffp{F_{k}}{x}(x)y^{k}  \in (\puiseux)[y] \, .
\]
Moreover, we define its \emph{derivative with respect to $y$} as
\[
\diffp{F}{y}(x,y) \coloneqq \sum_{k = 1}^{d}kF_{k}(x)y^{k-1} \in (\puiseux)[y] \, .
\]
\end{definition}

It is easy to check that derivation in $(\puiseux)[y]$ has the expected properties:
\begin{equation}\label{eq:partial_der}
\begin{aligned}
\diffp{(F + G)}{x}(x,y) &= \diffp{F}{x}(x,y) +\diffp{G}{x}(x,y) \, , \\ 
\diffp{(F + G)}{y}(x,y) &= \diffp{F}{y}(x,y) +\diffp{G}{y}(x,y) \, , \\
\diffp{(FG)}{x}(x,y) &= \diffp{F}{x}(x,y)G(x,y) + F(x,y)\diffp{G}{x}(x,y) \, , \\ 
\diffp{(FG)}{y}(x,y) &= \diffp{F}{y}(x,y)G(x,y) + F(x,y)\diffp{G}{y}(x,y) \, , \\
\diffp{}{x}\Bigl(\diffp{F}{y}\Bigr) (x,y) & = \diffp{}{y}\Bigl(\diffp{F}{x}\Bigr)(x,y) \, .
\end{aligned}
\end{equation}
In particular, \cref{eq:partial_der} implies that we can use the notation $\diffp[p,q]{F}{x,y}(x,y)$ for the element of $(\puiseux)[y]$ obtained from $F(x,y) \in (\puiseux)[y]$ by taking $p$ derivatives with respect to $x$ and $q$ derivatives with respect to $y$ (the result is the same for any order of derivation).

\begin{proof}[Proof of \cref{le:root_deriv}]
We can factorize $F(x,y)$ as $F(x,y) = (y- \pser(x))G(x,y)$ for some $G \in (\puiseux)[y]$. First, for every $p, q \ge 0$ such that $p + q \ge 1$ we want to prove the identity
\begin{equation}\label{eq:deriving_factor}
\begin{aligned}
\diffp[p,q]{F}{x,y}(x,y) = \diffp[p,q]{G}{x,y}(x,y)&\bigl( y - \pser(x) \bigr) + q\diffp[p,q-1]{G}{x,y}(x,y) \\ &- \sum_{\ell = 1}^{p}\binom{p}{\ell}\diffp[\ell]{\pser}{x}(x)\diffp[p-\ell,q]{G}{x,y}(x,y) \,
\end{aligned}
\end{equation}
(with the convention that the middle term vanishes if $q = 0$ and the last term vanishes if $p = 0$).
Indeed, for $(p,q) = (1,0)$ and $(p,q) = (0,1)$ we have
\begin{equation}\label{eq:deriving_factor_base}
\begin{aligned}
\diffp{F}{x}(x,y) &= \diffp{G}{x}(x,y)\bigl( y - \pser(x) \bigr) - \diffp{\pser}{x}(x)G(x,y) \, , \\
\diffp{F}{y}(x,y) &= \diffp{G}{y}(x,y)\bigl( y - \pser(x) \bigr) + G(x,y) \, \\
\end{aligned}
\end{equation}
as claimed. Moreover, note that for every $p \ge 1$ we have
\begin{align*}
&\diffp{}{x}\Bigl( \sum_{\ell = 1}^{p} \binom{p}{\ell} \diffp[\ell]{\pser}{x}(x)\diffp[p-\ell]{G}{x}(x,y) \Bigr) \\
&= \sum_{\ell = 1}^{p} \binom{p}{\ell} \diffp[\ell+1]{\pser}{x}(x)\diffp[p-\ell]{G}{x}(x,y) + \sum_{\ell = 1}^{p} \binom{p}{\ell} \diffp[\ell]{\pser}{x}(x)\diffp[p+1-\ell]{G}{x}(x,y) \\
&= \sum_{\ell = 2}^{p+1} \binom{p}{\ell - 1} \diffp[\ell]{\pser}{x}(x)\diffp[p+1-\ell]{G}{x}(x,y) + \sum_{\ell = 1}^{p} \binom{p}{\ell} \diffp[\ell]{\pser}{x}(x)\diffp[p+1-\ell]{G}{x}(x,y) \\
&= -\diffp{\pser}{x}(x)\diffp[p]{G}{x}(x,y) + \sum_{\ell = 1}^{p + 1}\binom{p+1}{\ell}\diffp[\ell]{\pser}{x}(x)\diffp[p+1-\ell]{G}{x}(x,y) \, .
\end{align*}
Hence, by induction, for every $p \ge 1$ we get
\begin{align*}
&\diffp[p + 1]{F}{x}(x,y) = \diffp{}{x}\Bigl( \diffp[p]{G}{x}(x,y)\bigl( y - \pser(x) \bigr) - \sum_{\ell = 1}^{p} \binom{p}{\ell} \diffp[\ell]{\pser}{x}(x)\diffp[p-\ell]{G}{x}(x,y)\Bigr) \\
&=  \diffp[p+1]{G}{x}(x,y)\bigl( y - \pser(x) \bigr) - \sum_{\ell = 1}^{p + 1}\binom{p+1}{\ell}\diffp[\ell]{\pser}{x}(x)\diffp[p+1-\ell]{G}{x}(x,y)
\end{align*}
and~\cref{eq:deriving_factor} is true for whenever $q = 0$ or $(p,q) = (0,1)$. By using the induction once more, we obtain that $\diffp[p,q+1]{F}{x,y}(x,y)$ is equal to
\begin{align*}
\diffp{}{y}\Bigl( \diffp[p,q]{G}{x,y}(x,y)&\bigl( y - \pser(x) \bigr) + q\diffp[p,q-1]{G}{x,y}(x,y) \\ 
&- \sum_{\ell = 1}^{p}\binom{p}{\ell}\diffp[\ell]{\pser}{x}(x)\diffp[p-\ell,q]{G}{x,y}(x,y) \Bigr) \, ,
\end{align*}
which is equal to
\begin{align*}
\diffp[p,q+1]{G}{x,y}(x,y)&\bigl( y - \pser(x) \bigr) + (q + 1)\diffp[p,q]{G}{x,y}(x,y) \\
&- \sum_{\ell = 1}^{p}\binom{p}{\ell}\diffp[\ell]{\pser}{x}(x)\diffp[p-\ell,q+1]{G}{x,y}(x,y)
\end{align*}
and~\cref{eq:deriving_factor} is true for all pairs $(p,q)$ such that $p + q \ge 1$. In particular, we obtain the identity
\begin{equation}\label{eq:deriving_factor_root}
\begin{aligned}
&\diffp[p,q]{F}{x,y}(x,\pser(x)) \\
&= q\diffp[p,q-1]{G}{x,y}(x,\pser(x)) - \sum_{\ell = 1}^{p}\binom{p}{\ell}\diffp[\ell]{\pser}{x}(x)\diffp[p-\ell,q]{G}{x,y}(x,\pser(x)) \, .
\end{aligned}
\end{equation}
Let 
\[
A_{p}(x) \coloneqq \Bigl(\diffp[p]{\pser}{x}(x)\Bigr)\bigl( G(x,\pser(x)) \bigr)^{2p - 1}
\]
for all $p \ge 1$,
\[
B_{p,q}(x) \coloneqq \Bigl(\diffp[p,q]{G}{x,y}(x,\pser(x))\Bigr)\bigl( G(x,\pser(x)) \bigr)^{2p +q - 1} \, 
\]
for all $p,q \ge 0$ such that $p + q \ge 1$, and 
\[
C_{p,q}(x) \coloneqq \Bigl(\diffp[p,q]{F}{x,y}(x,\pser(x))\Bigr)\bigl( G(x,\pser(x)) \bigr)^{2p +q - 2} \,
\]
for all $p,q \ge 0$ such that $2p + q \ge 2$. For every $p \ge 1$ we take $q = 0$, multiply~\cref{eq:deriving_factor_root} by $G(x,\pser(x))^{2p - 2}$ and obtain
\begin{equation}\label{eq:formula_for_deriv_1}
A_{p}(x) = C_{p,0}(x) + \sum_{\ell = 1}^{p - 1}\binom{p}{\ell}A_{\ell}(x)B_{p - \ell,0}(x) \, .
\end{equation}
Similarly, for every $q \ge 1$ we use~\cref{eq:deriving_factor_root} for the tuple $(p, q+1)$, multiply this equality by $G(x,\pser(x))^{2p + q - 1}$, and obtain the formula
\begin{equation}\label{eq:formula_for_deriv_2}
B_{p,q}(x) = \begin{cases}
\frac{1}{q+1}\Bigl( C_{p,q+1}(x) + \sum_{\ell = 1}^{p}\binom{p}{\ell}A_{\ell}(x)B_{p - \ell,q+1}(x) \Bigr) &\text{if $p > 0$} \, , \\
\frac{1}{q+1}C_{0,q+1}(x) &\text{otherwise.}
\end{cases}
\end{equation}
Moreover, we note that~\cref{eq:deriving_factor_base} gives the identity
\begin{equation}\label{eq:factor_as_derivative}
\diffp{F}{y}(x,\pser(x)) = G(x,\pser(x)) \, .
\end{equation}
In particular, we have
\begin{equation}\label{eq:auxpol}
C_{p,q}(x) \coloneqq \Bigl(\diffp[p,q]{F}{x,y}(x,\pser(x))\Bigr)\Bigl( \diffp{F}{y}(x,\pser(x)) \Bigr)^{2p +q - 2} \, .
\end{equation}
By~\cref{eq:factor_as_derivative}, to prove the claim we want to prove that for every $k \ge 1$ there exists a polynomial $\rootpol_{k} \in \Z[x_{1}, \dots, x_{k(k+3)/2}]$ of degree at most $2k-1$ such that
\[
A_{k}(x) = \rootpol_{k}\Bigl( \bigl(\diffp[p,q]{F}{x,y}(x,\pser(x))\bigr)_{1 \le p+q \le k}\Bigr) \, .
\]
To do so, we use an auxiliary family of polynomials. More precisely, we will show that for every $k,l \ge 0$ such that $k + l \ge 1$ there exists a polynomial $\auxrootpol_{k,l} \in \Z[x_{1}, \dots, x_{(k+l+1)(k+l+4)/2}]$ of degree at most $2k + l$ such that
\[
B_{k,l}(x) = \auxrootpol_{k,l}\Bigl( \bigl(\diffp[p,q]{F}{x,y}(x,\pser(x))\bigr)_{1 \le p+q \le k + l +1}\Bigr) \, .
\]
We prove the existence of $\rootpol_{k}$ and $\auxrootpol_{k,l}$ by induction over $k$. For $k = 0$ we have $B_{0,l}(x) = \frac{1}{l+1}C_{0,l+1}(x)$ and the claim follows from~\cref{eq:auxpol}. For $k = 1$ we have $A_{1}(x) = C_{1,0}(x)$, and thus $\rootpol_{1}$ exists as claimed. Moreover, we have 
\begin{align*}
B_{1,l}(x) &= \frac{1}{l+1}(C_{1,l+1}(x) + A_{1}(x)B_{0,l+1}(x)) \\
&= \frac{1}{l+1}(C_{1,l+1}(x) + \frac{1}{l+2}C_{1,0}(x)C_{0,l+2}(x))
\end{align*}
and therefore $\auxrootpol_{1,l}$ exist. For every $\ell \ge 1$ let $Z_{\ell} \in \puiseux^{\ell(\ell+3)/2}$ be defined as
\[
Z_{\ell} \coloneqq \Bigl( \bigl(\diffp[p,q]{F}{x,y}(x,\pser(x))\bigr)_{1 \le p+q \le \ell}\Bigr) \, .
\]
By induction, for every $k \ge 2$ we have
\begin{align*}
A_{k}(x) &= C_{k,0}(x) + \sum_{i = 1}^{k-1}\binom{k}{i}\rootpol_{i}(Z_{i})\auxrootpol_{k - i,0}(Z_{k - i + 1}) \, , \\
B_{k,l}(x) &= \frac{1}{l+1}C_{k,l+1}(x) + \frac{1}{l+1}\sum_{i = 1}^{k}\binom{k}{i}\rootpol_{i}(Z_{i})\auxrootpol_{k - i,l+1}(Z_{k - i + l + 2}) \, .
\end{align*}
The claim follows by computing the degrees of the resulting polynomials $\rootpol_{k},\auxrootpol_{k,l}$.
\end{proof}

\end{document}